\documentclass[12pt]{article}

\usepackage{amsmath,amssymb,amsthm,mathrsfs,amsfonts,amscd,esint,cite}

\usepackage{sidecap}
\usepackage{float}
\usepackage{extarrows}
\usepackage{verbatim}
\usepackage{hyperref}
\usepackage[usenames,dvipsnames]{xcolor}
\usepackage{enumitem}
\usepackage{mathtools}

\newtheorem{thm}{Theorem}[section]

\newtheorem{lemma}[thm]{Lemma}

\newtheorem{prop}[thm]{Proposition}

\newtheorem{remark}[thm]{Remark}

\def\XXint#1#2#3{{\setbox0=\hbox{$#1{#2#3}{\int}$}
		\vcenter{\hbox{$#2#3$}}\kern-.5\wd0}}

\def \R {\mathbb R}
\def \N {\mathbb N}

\numberwithin{equation}{section}

\begin{document}
	
	\title{Well-posedness and scattering of odd solutions for the defocusing INLS in one dimension
 \thanks{This work is supported by NSFC under the grant Nos. (12301090, 12075102) } }
	\author{Zhi-Yuan Cui, Yuan Li\thanks{Corresponding author: li\_yuan@lzu.edu.cn (Yuan Li)} and Dun Zhao 
 \\{\small School of Mathematics and Statistics, Lanzhou University, Lanzhou 730000, China}}
	\date{}
	\maketitle
	
	\begin{abstract}

	We consider the defocusing inhomogeneous nonlinear Schr\"{o}dinger equation
  $i\partial_tu+\Delta u= |x|^{-b}|u|^{\alpha}u,$
    where $0<b<1$ and $0<\alpha<\infty$. This problem has been extensively studied for initial data in $H^1(\R^N)$ with $N\geq 2$.  However, in the one-dimensional setting, due to the difficulty in dealing with the singularity factor $|x|^{-b}$,  the well-posedness and scattering  in $H^1(\R)$  are scarce, and almost known results have been established in $H^s(\R)$ with $s<1$. In this paper, we focus on the odd initial data in $H^1(\R)$.  For this case, we establish local well-posedness for $0<\alpha<\infty$, as well as global well-posedness and scattering for $4-2b<\alpha<\infty$, which corresponds to the mass-supercritical  case. The key ingredient is the application of the one-dimensional Hardy inequality for odd functions to overcome the singularity induced by $|x|^{-b}$.  Our proof is based on the Strichartz estimates  and employs the concentration-compactness/rigidity method developed by Kenig-Merle  as well as the technique for handling initial data living far from the origin, as proposed by Miao-Murphy-Zheng.  Our results fill a gap in the theory of well-posedness and energy scattering for the inhomogeneous nonlinear Schr\"{o}dinger equation in one dimension.



\vspace{0.2cm}
		\vspace{0.2cm}
		
	\end{abstract}

\section{Introduction and Main Results}
In this paper, we explore the one-dimensional defocusing inhomogeneous nonlinear Schr\"{o}dinger equation (INLS)
\begin{equation}\label{eqinls}
	\left\{\begin{aligned}
	&i\partial_tu+\Delta u=|x|^{-b}|u|^{\alpha}u,\\
	&u|_{t=0}=u_0\in H^1(\R),
	\end{aligned}\right.
	\end{equation}
where $u=u(t,x)$ is a complex-valued function defined in time-space $\R\times\R$, with parameters $0<b<1$ and $0<\alpha<\infty$. This kind of problem \eqref{eqinls} arises in nonlinear optical systems with spatially dependent interactions, see, for example, \cite{2007PRL}.  For more detailed physical insights, we refer to the works of Gill \cite{Gill2000}, and Liu and Tripathi \cite{LT1994}.

The equation \eqref{eqinls} is invariant under the scaling transformation defined by
$$u_\lambda(t,x):=\lambda^\frac{2-b}{\alpha}u(\lambda^2t,\lambda x),$$
where $\lambda>0$. It preserves the norm in the homogeneous Sobolev space $\dot{H}^{s_c}(\R)$, as given by:
\begin{align}\label{def:scaling}   \|u_\lambda(0,\cdot)\|_{\dot{H}^{s_c}}=\|u_0\|_{\dot{H}^{s_c}},\qquad s_c:=\frac{1}{2}-\frac{2-b}{\alpha}.
\end{align}
 If $s_c<0$, problem \eqref{eqinls} is referred to as mass sub-critical; if $s_c=0$, it is called mass critical;  if $s_c>0$, it is mass super-critical. In particular, it is always energy subcritical in the one-dimensional since $s_c<\frac{1}{2}$.

Equation \eqref{eqinls} also has the following conserved quantities:

Mass:
\begin{equation}\label{M}
    M[u(t)]=\int |u(t,x)|^2dx=M[u_0];
    \end{equation}

Energy:
\begin{equation}\label{E}
  E[u(t)]=\frac{1}{2}\int |\nabla u(t,x)|^2dx+\frac{1}{\alpha+2}\int |x|^{-b}|u(t,x)|^{\alpha+2}dx=E[u_0].
\end{equation}

When $b=0$, equation \eqref{eqinls} becomes the classical nonlinear Schr\"odinger equation, which has been extensively studied in the literature. For more details, we refer to \cite{FXC2011, HR2008, DHR2008, CKSTT2008, Visan2007, RV2007, Dodson2019, Cz2003, Vis2008, Bourgain1999} and the references therein. In contrast, when $b\neq0$,  much less is known about equation \eqref{eqinls}, and it has garnered increasing attention in recent years. In particular, results of well-posedness  and scattering have been obtained by many authors. For some relevant studies, see \cite{Farah2016, GS2008,FG2017,Guz2017, Dinh2019,FG2020,AT2021,Campos2021,Dinh2021-1,Dinh2021-2,Dinh2021-3,AK2021,KLS2021,MMZ2021,CC2022,CFGM2022} and the references therein.

 We briefly review some developments for well-posedness in one dimension.  By using the energy method, Genoud-Stuart\cite{GS2008} and Farah \cite{Farah2016} obtained the well-posedness of equation \eqref{eqinls} for $4-2b<\alpha<\infty$ and $0<b<1$ in $H^1(\R)$ (also for $\frac{4-2b}{N}<\alpha<\frac{4-2b}{N-2}$ and $0<b<2$ in $H^1(\R^{N})$ for $N\geq2$).  Their results show that solutions satisfy $u\in C((-T,T),H^1(\R))$. 
 Other one-dimensional well-posedness results based on Kato's method (using Strichartz estimates), such as in \cite{Guz2017,AT2021,AK2021}, have focused on the initial data in $H^s(\R)$, where $s<1$. The restriction $s<1$ arises due to the challenge of handling the singular factor $|x|^{-b}$. More precisely, Guzm\'an \cite{Guz2017} explored the well-posedness in $L^2(\R)$  for $0<b<1$ and $\alpha<4-2b$; The author also investigated the well-posedness in $H^s(\R)$ for $\max\{0,s_c\}<s\leq \frac{1}{2}$, $0<b<\frac{1}{3}$ and $0<\alpha<\frac{4-2b}{1-2s}$. Aloui and Tayachi \cite{AT2021} established local well-posedness in $H^s(\R)$ for $0\leq s <\frac{1}{2}$,  $0<b<1-2s$ and $0<\alpha\leq \frac{4-2b}{1-2s}$.  Subsequently, An and Kim \cite{AK2021} analyzed local well-posedness in $H^s(\R)$ for $0\leq s <1$,  $0<b<1-s$ and $0<\alpha< \alpha_s$, where $\alpha_s=\frac{4-2b}{1-2s}$ if $0\leq s<\frac{1}{2}$, or $\alpha_s=\infty$ if $1>s\geq\frac{1}{2}$. These works provide that solutions belong to $L^q_{loc}((-T,T),W^{s,r}(\R))$ for any $L^2$-admissible pair $(q,r)$. Furthermore, Dinh points out that in contrast to the case $N\geq 2$, where solutions satisfy $u\in L^q_{loc}((-T,T),W^{1,r}(\R^N))$, a similar result cannot be expected in the one-dimensional case when using Strichartz estimates (see Remark 3.2 or the sentences before Remark 1.1 in \cite{Dinh2021-1}).

For the scattering results of INLS in $H^1(\R^N)$, as far as we know, only the cases for $N\geq 2$ have been studied, and for the case of $N=1$, the related result is still lacking in the literature. To our knowledge, the energy scattering for the focusing INLS was first established by Farah and Guzm\'an \cite{FG2017} with $0<b<\frac{1}{2},~\alpha=2,~N=3$ and radial data.  For general results on the scattering in $H^1$, one can see \cite{Dinh2021-2,FG2020,Campos2021,CC2022,MMZ2021,Dinh2021-3,CFGM2022}  and the references therein.  In particular, for $N\geq2,~\frac{4-2b}{N}<\alpha<\frac{4-2b}{N-2}$ and $0<b<\min\{2,\frac{N}{2}\}$, \cite{CFGM2022,Dinh2021-3} obtained the energy scattering in $H^1$ without the radial restriction. So far, this range of $(N,\alpha, b)$ is the widest. For the defocusing case, we refer to \cite{Dinh2019, Dinh2021-2}; for the energy critical case, see  \cite{GM2021,CHL2020,Park2024}. In addition, when $N=1$,  Dinh \cite{Dinh2021-1} obtained a decay estimate for equation \eqref{eqinls} using the pseudoconformal conservation law in the weighted $L^2$-space.

In this paper, we will investigate the well-posedness and scattering  with odd initial data in $H^1(\mathbb{R})$. Under this condition, the solution $u(t)$ to \eqref{eqinls} will also be odd. Specifically, we have $u(t,0)\equiv0$. Hence, using the one-dimensional Hardy inequality for odd functions, we can deal with the singularity term $|x|^{-(b+1)}$, which comes from the first derivative of the nonlinear term corresponding to the derivative part of $H^1$ (see \eqref{est:line:3} and \eqref{eq th2}). Moreover, under the oddness assumption, our results answer the left issues in Dinh's works \cite{Dinh2019} and \cite{Dinh2021-1} (see the comments of Theorem 3 and 4 in \cite{Dinh2019} and Remark 3.2 in \cite{Dinh2021-1}).

We mention that while \cite{Farah2016} establishes the well-posedness, their proof avoids Strichartz estimates. Consequently, the question of whether the solutions to \eqref{eqinls} belong to $L^q_{loc}((-T,T),W^{1,r}(\R))$ for any $L^2$-admissible pair $(q,r)$ remains unknown. This integrability in time-space Lebesgue spaces plays an important role in proving the energy scattering for INLS. Therefore, we will establish a well-posedness estimate using Kato's method. Our first result can be stated as:

\begin{thm}\label{T1}
Let $0<b<1$ and $0<\alpha<\infty$.

1) (Local well-posedness) Let $u_0\in H^1(\R)$ be an odd function. Then, there exists a time $T=T(u_0)>0$ and a unique odd solution $u(t)$ to \eqref{eqinls} that satisfies
   $$u\in C((-T,T),H^1(\R))\cap L^q_{loc}((-T,T),W^{1,r}(\R))$$
for any $L^2$-admissible pair $(q,r)$. Moreover, the continuous dependence on the initial data is valid.

2) (Global theory of small data) Assume that $4-2b<\alpha<\infty$. Let $u_0\in H^1(\R)$ be an odd function such that $\|u_0\|_{H^1}\leq A$ for some constant $A>0$. Then, there exists  $\delta=\delta(A)>0$ such that if $\|e^{it\Delta}u_0\|_{S(\dot{H}^{s_c})}<\delta$, the corresponding solution to \eqref{eqinls} with initial data $u_0$ is global in time and satisfies
  $$\|u\|_{S(\dot{H}^{s_c})}\leq 2\|e^{it\Delta}u_0\|_{S(\dot{H}^{s_c})},\quad \|u\|_{S(L^2)}+\|\nabla u\|_{S(L^2)}\leq 2C\|u_0\|_{H^1}.$$

3) (Scattering condition) Assume that $4-2b<\alpha<\infty$. Let $u$ be a global odd solution to \eqref{eqinls}. If
  $$\|u\|_{L_t^\infty(\R,H_x^1)}\leq A \quad \text{and} \quad \|u\|_{S(\dot{H}^{s_c})}<\infty,$$
then the solution $u$ scatters in $H^1(\R)$ in both directions, i.e., there exist $u_{\pm}$ in $H^1(\R)$ such that
 $$\lim_{t\rightarrow\pm \infty}\|u(t)-e^{it\Delta}u_{\pm}\|_{H^1(\R)}=0.$$

4) (Existence of wave operator) Suppose $4-2b<\alpha<\infty$ and let $\phi\in H^1(\R)$ be an odd function. Then there exists $u_0^+\in H^1(\R)$ such that the corresponding solution $u$ to \eqref{eqinls} is global in $H^1(\R)$ and satisfies
$$M[u]=M[\phi],\qquad E[u]=\frac{1}{2}\|\nabla\phi\|^2_{L^2},\qquad \lim_{t\rightarrow+\infty}\|u(t)-e^{it\Delta}\phi\|_{H^1}=0.$$

\end{thm}

{\bf  Remark:} Similar results can be established for the one dimensional focusing INLS:
\begin{equation*}
        i\partial_tu+\Delta u=-|x|^{-b}|u|^{\alpha}u.
\end{equation*}

To achieve the scattering solution for the INLS  \eqref{eqinls} in one dimension, a primary challenge is the lack of decay estimates in $H^1(\R)$. However, the decay estimates can be obtained by the odd assumption, see our work \cite{CLZ2024}. With the aid of these decay estimates, along with techniques introduced by Miao, Murphy, and Zheng \cite{MMZ2021}, we can ultimately establish the scattering result as following:

\begin{thm}\label{T2}
Assume that $0<b<1 $ and $4-2b<\alpha<\infty$. If $u_0\in H^1(\R)$ is an odd function, then the corresponding odd solution $u(t)$ to \eqref{eqinls} is global and scatters in $H^1(\R)$.
\end{thm}

We point out that in the literature on energy scattering, the parameter $b$  has always been restricted to $0<b<\min\{2,\frac{N}{2}\}$ (see \cite{Campos2021,CFGM2022,CC2022,Dinh2021-2,Dinh2021-3}), which is not consistent with the existence condition on $b$ ($0<b<\min\{2,N\}$, as presented in \cite{Farah2016}). In previous works, the restriction stems from the techniques (such as Hardy inequality or H\"{o}lder's inequality in Lorentz spaces) used to handle the singularity arising from the term $|x|^{-b}$. In Theorem \ref{T2}, the condition $0<b<1$ aligns with the existence criteria  presented in \cite{Farah2016}. In fact, by the Lemma \ref{Lem Hardy}, the oddness ensures that the nonlinear estimates hold for $0<b<1$. In some sense, for nonlinearities of the form $|x|^{-b}|u|^\alpha u$, the $H^1$-solutions with $u(t,0)=0$ mitigate the singularity arising from the factor $|x|^{-b}$ at the origin.

This paper is organized as follows: In Section 2, we introduce some useful lemmas and inequalities. Section 3 is dedicated to the proof of Theorem \ref{T1}. In Section 4, we establish a stability result. Section 5 focuses on constructing the minimal blowup solution, which is essential to our argument. Finally, in Section 6, we present the proof of Theorem \ref{T2}.

Here we list some notations that will be used throughout this paper:

- We write $A\lesssim B$ to indicate that there exists a constant $C>0$ such that $A\leq CB$, where $C$ does not depend on $A$ and $B$. If the constant $C$ depends on a parameter $M$, we use $A\lesssim_M B$.

- The notation $\nabla f$ refers to the derivative of the function $f$ with respect to the spatial variable $x$.

- We use the standard notation for time-space Lebesgue spaces $L^q_tL^r_x$. Furthermore, we denote the homogeneous Sobolev spaces as $\dot{W}^{s,r}$ and the inhomogeneous Sobolev spaces as $W^{s,r}$. In particular, when $r=2$, we use the notations $\dot{W}^{s,2}=\dot{H}^s$ and $W^{s,2}=H^s$.  Additionally, when $s = 0$, it follows that $\dot{H}^0=L^2$.  If necessary, subscripts may be used to clarify which variable is being referenced.

- The notation $p'$ denotes the H\"{o}lder conjugate of $p\geq1$.

\section{Preliminaries}

~~~~ In this section, we give some basic lemmas and inequalities.

We employ the standard Littlewood-Paley projections $P_{\leq N}$ which are defined as Fourier multipliers. These multipliers are associated with a smooth cutoff to the region $\{|\xi|\leq N\}$.  In particular, we require some fundamental results, such as the Bernstein estimate
\begin{equation}
    \label{eq Bern}
\||\nabla|^sP_{\leq N}f\|_{L^2_x}\leq N^s\|f\|_{L^2_x} 
\end{equation}
and the fact that $P_{\leq N}f\rightarrow f$ strongly in $H^1$ as $N\rightarrow\infty$.

We also recall the Strichartz estimates associated with the free Schr\"odinger operator. For $s\in (-\frac{1}{2},\frac{1}{2})$, a pair $(q,r)$ is called $\dot{H}^s$-admissible if it satisfies $\frac{2}{q}=\frac{1}{2}-\frac{1}{r}-s$ with $\frac{2}{1-2s}\leq r<\infty$.  Let $I\subset\R$ and $\mathcal{A}_s$ be the set of $\dot{H}^s$-admissible pairs. We introduce the Strichartz norm
$$\|u\|_{S(I,\dot{H}^s)}=\sup_{(q,r)\in \mathcal{A}_s}\|u\|_{L^q_t(I)L^r_x(\R)}.$$
We also define the dual Strichartz norm by
$$\|u\|_{S'(I,\dot{H}^{-s})}=\inf_{(q,r)\in \mathcal{A}_{-s}}\|u\|_{L^{q'}_t(I)L^{r'}_x(\R)}.$$
When $I=\R$, we simply denote $\|u\|_{S(\dot{H}^s)}:=\|u\|_{S(I,\dot{H}^s)}$ and $\|u\|_{S'(\dot{H}^{-s})}:=\|u\|_{S'(I,\dot{H}^{-s})}$.

We have the following Strichartz estimates (see \cite{Foschi2005,KT1998,Cz2003}).
\begin{lemma}\label{str est}
Assume that $s\geq0$. Let $f\in \dot{H}^s$ and $g\in L^{q'}_t(I)L^{r'}_x(\R)$ for some $(q,r)\in\mathcal{A}_{-s}$, then the following hold:
\begin{equation*}
      \begin{split}
          \|e^{it\Delta}f\|_{S(\dot{H}^s)}&\lesssim\|f\|_{\dot{H}^s},\\
      \left\|\int^t_0e^{i(t-t')\Delta}g(t')dt'\right\|_{S(I,\dot{H}^s)}&\lesssim\|g\|_{S'(I,\dot{H}^{-s})}.
       \end{split}
   \end{equation*}
\end{lemma}

Next, we present the following nonlinear estimates. A similar argument can be found in \cite{CFGM2022, Dinh2021-3}.

\begin{lemma}\label{lem nlest}
 Assume that $0<b<1$ and $4-2b<\alpha<\infty$. Let $u,v\in H^1(\R)$ be odd functions. Then there exists a sufficiently small $\theta\in(0,\alpha)$ such that
\begin{align}\label{est:line:1}
     \||x|^{-b}|u|^\alpha v\|_{S'(I,\dot{H}^{-s_c})}&\lesssim\|u\|^\theta_{L^\infty_tH^1_x}\|u\|^{\alpha-\theta}_{S(I,\dot{H}^{s_c})}\|v\|_{S(I,\dot{H}^{s_c})},\\\label{est:line:2}
           \||x|^{-b}|u|^\alpha v\|_{S'(I,L^2)}&\lesssim\|u\|^\theta_{L^\infty_tH^1_x}\|u\|^{\alpha-\theta}_{S(I,\dot{H}^{s_c})}\|v\|_{S(I,L^2)},\\\label{est:line:3}
           \|\nabla(|x|^{-b}|u|^\alpha u)\|_{S'(I,L^2)}&\lesssim\|u\|^\theta_{L^\infty_tH^1_x}\|u\|^{\alpha-\theta}_{S(I,\dot{H}^{s_c})}\|\nabla u\|_{S(I,L^2)},
\end{align}
where $I\subset \R$ and $s_c$ is defined as in  \eqref{def:scaling}.
\end{lemma}

To establish local well-posedness, we also provide another version of the nonlinear estimates that depend on time, allowing for $0<\alpha\leq \infty$.

\begin{lemma}\label{lem Tnlest}
    Assume that $0<b<1$ and $0<\alpha<\infty$. Let $u,v\in H^1(\R)$ be odd functions. Then there exists  $\theta\in(\frac{1}{2},\frac{3}{4})$ such that
    \begin{align}\label{eq tl2}
    \||x|^{-b}|u|^\alpha v\|_{S'(I,L^2)}&\lesssim T^\theta\|u\|^\alpha_{L^\infty_t(I)H^1_x}\|v\|_{S(I,L^2)},\\
    \label{eq th2}
           \|\nabla(|x|^{-b}|u|^\alpha u)\|_{S'(I,L^2)}&\lesssim T^\theta\|u\|^\alpha_{L^\infty_t(I)H^1_x}\|\nabla u\|_{S(I,L^2)},
\end{align}
for $I=[0,T]$.
\end{lemma}

To prove Lemma \ref{lem nlest} and Lemma \ref{lem Tnlest}, we need the following one-dimensional Hardy inequality.

\begin{lemma}\label{Lem Hardy}
    If $1<p<\infty$ and $u\in H^1(\R)$ is an odd function, then \begin{equation}\label{Hardy}
        \int^\infty_0\frac{|u(r)|^p}{r^p}dr\leq\left(\frac{p}{p-1}\right)^p\int^\infty_0|\nabla u(r)|^pdr.
    \end{equation}
\end{lemma}

\begin{proof}
Since $u\in H^1(\R)$ is an odd function, there is a function $v\in C(\R)$ such that $v=u$ holds almost everywhere, and $v|_{(0,\infty)}$ is locally absolutely continuous on $(0,\infty)$ with $\liminf_{r\rightarrow0}|v(r)|=0$. Therefore, from the Hardy inequality for absolutely continuous functions (see the inequality (1) or 
Theorem 1 of \cite{FLW2022}), inequality \eqref{Hardy} holds. A similar result can be founded in Exercise 8.8 in \cite{Brezis}.
\end{proof}

\begin{proof}[Proof of Lemma \ref{lem nlest}]
 Following a similar argument as in \cite[Lemma 2.1]{CFGM2022}, we define $B=B(0,1)=\{x\in\R:|x|\leq 1\}$ and let $B^C$ denote its complement. We introduce the following parameters
\begin{equation*}\begin{split}
        \hat{q}=\frac{4\alpha(\alpha+2-\theta)}{\alpha(\alpha+2b)-\theta(\alpha-4+2b)},~~~\hat{r}=\frac{\alpha(\alpha+2-\theta)}{\alpha(1-b)-\theta(2-b)},\\
        \tilde{a}=\frac{2\alpha(\alpha+2-\theta)}{\alpha(\alpha-\theta-1+2b)-(4-2b)(1-\theta)},~~~\hat{a}=\frac{2\alpha(\alpha+2-\theta)}{\alpha+4-2b}.
        \end{split}
\end{equation*}
For a sufficiently small $\theta>0$, it is straightforward to verify that $(\hat{q},\hat{r})\in \mathcal{A}_0$, $(\hat{a},\hat{r})\in\mathcal{A}_{s_c}$ and $(\tilde{a},\hat{r})\in\mathcal{A}_{-s_c}$. These exponents satisfy the scaling relations:
\begin{equation*}
    \frac{1}{\tilde{a}'}=\frac{\alpha-\theta}{\hat{a}}+\frac{1}{\hat{a}}~~~~\text{and}~~~~ \frac{1}{\hat{p}'}=\frac{\alpha-\theta}{\hat{a}}+\frac{1}{\hat{p}}.
\end{equation*}
To prove \eqref{est:line:1}, taking $r_1\in(\frac{1}{\theta},\infty)$ and $\gamma$ such that
\begin{equation*}
    \frac{1}{\hat{r}'}=\frac{1}{\gamma}+\frac{1}{r_1}+\frac{\alpha-\theta}{\hat{r}}+\frac{1}{\hat{r}},
\end{equation*}
by applying H\"{o}lder's inequality, we have
\begin{equation*}
        \||x|^{-b}|u|^\alpha v\|_{L^{\hat{r}'}_x(A)}\lesssim\||x|^{-b}\|_{L^\gamma(A)}\|u\|^\theta_{L^{\theta r_1}_x}\|u\|^{\alpha-\theta}_{L^{\hat{r}}_x}\|v\|_{L^{\hat{r}}_x},
\end{equation*}
where $A$  represents either $B$ or $B^C$.

Indeed, observe that
\begin{equation*}
        \frac{1}{\gamma}-b=\frac{\theta(2-b)}{\alpha}-\frac{1}{r_1}.
    \end{equation*}
If $A=B$, we choose $r_1$ such that $\theta r_1>\frac{\alpha}{2-b}$. This ensures $\frac{1}{\gamma}-b>0$ and $|x|^{-b}\in L^\gamma(B)$. If $A=B^C$, we choose $r_1$ so that $\theta r_1=2$. In this case, $\frac{1}{\gamma}-b<0$, which implies $|x|^{-b}\in L^\gamma(B^C)$. In both cases, $\||x|^{-b}\|_{L^\gamma(A)}<\infty$.

Using the Sobolev embedding $H^1\subset L^{\theta r_1}$ and H\"{o}lder's inequality in time, we derive
\begin{equation*}
        \||x|^{-b}|u|^\alpha v\|_{L^{\tilde{a}'}_tL^{\hat{r}'}_x(A)}\lesssim\|u\|^\theta_{L^\infty_tH^1_x}\|u\|^{\alpha-\theta}_{L^{\hat{a}}_tL^{\hat{r}}_x}\|v\|_{L^{\hat{a}}_tL^{\hat{r}}_x}.
    \end{equation*}
This implies that \eqref{est:line:1} holds.

Similarly, \eqref{est:line:2} can be established by 
\begin{equation*}
    \||x|^{-b}|u|^\alpha v\|_{L^{{\hat{q}}'}_tL^{\hat{r}'}_x(A)}\lesssim\|u\|^\theta_{L^\infty_tH^1_x}\|u\|^{\alpha-\theta}_{L^{\hat{a}}_tL^{\hat{r}}_x}\|v\|_{L^{\hat{q}}_tL^{\hat{r}}_x}.
\end{equation*}

To prove \eqref{est:line:3}, we use the chain rule to deduce
\begin{equation*}
    \left\|\nabla\left(|x|^{-b}|u|^\alpha u\right)\right\|_{L^{\hat{q}'}_tL^{\hat{r}'}_x(A)}\lesssim \||x|^{-b}|u|^\alpha |\nabla u|\|_{L^{\hat{q}'}_tL^{\hat{r}'}_x(A)}+ \||x|^{-b-1}|u|^{\alpha+1}\|_{L^{\hat{q}'}_tL^{\hat{r}'}_x(A)}=N_1+N_2.
\end{equation*}
By the similar argument as used for \eqref{est:line:1} and \eqref{est:line:2}, we have
 \begin{equation*}       N_1\lesssim\|u\|^\theta_{L^\infty_tH^1_x}\|u\|^{\alpha-\theta}_{L^{\hat{a}}_tL^{\hat{r}}_x}\|\nabla u\|_{L^{\hat{q}}_tL^{\hat{r}}_x},~~~~\text{and}~~~~
N_2\lesssim\|u\|^\theta_{L^\infty_tH^1_x}\|u\|^{\alpha-\theta}_{L^{\hat{a}}_tL^{\hat{r}}_x}\left\|\frac{u}{x}\right\|_{L^{\hat{q}}_tL^{\hat{r}}_x}.
\end{equation*}
From Lemma \ref{Lem Hardy}, we obtain
\begin{equation*}
 \left\|\frac{u}{x}\right\|_{L^{\hat{q}}_tL^{\hat{r}}_x}\lesssim\|\nabla u\|_{L^{\hat{q}}_tL^{\hat{r}}_x},
\end{equation*}
and
\begin{equation*}
   \left\|\nabla\left(|x|^{-b}|u|^\alpha u\right)\right\|_{L^{\hat{q}'}_tL^{\hat{r}'}_x(A)}\lesssim\|u\|^\theta_{L^\infty_tH^1_x}\|u\|^{\alpha-\theta}_{L^{\hat{a}}_tL^{\hat{r}}_x}\|\nabla u\|_{L^{\hat{q}}_tL^{\hat{r}}_x}.
\end{equation*}
By combining these estimates, we prove \eqref{est:line:3}. Thus, we complete the proof of Lemma \ref{lem nlest}.
\end{proof}



\begin{remark}
In this proof, the oddness is specifically used to address the singularity term $|x|^{-(b+1)}$ of $N_2$. Obviously, for other arbitrary functions in $H^1(\R)$ without oddness, one even can not expect that the part containing $|x|^{-(b+1)}$ is boundedness. On the other hand, we demonstrate a variant of \eqref{est:line:2} as following
     \begin{equation}\label{rem nlest}
         \||x|^{-b}|u|^{\alpha-1}vw\|_{S'(I,L^2)}\lesssim\|u\|^\theta_{L^\infty_tH^1_x}\|u\|^{\alpha-1-\theta}_{S(\dot{H}^{s_c})}\|v\|_{S(\dot{H}^{s_c})}\|w\|_{S(L^2)},
     \end{equation}
which is derived by using the following relations
\begin{equation*}
         \frac{1}{\hat{r}'}=\frac{1}{\gamma}+\frac{1}{r_1}+\frac{\alpha-1-\theta}{\hat{r}}+\frac{1}{\hat{r}}+\frac{1}{\hat{r}},
         ~~\text{and}~~\frac{1}{\hat{q}'}=\frac{1}{\infty}+\frac{1}{\infty}+\frac{\alpha-1-\theta}{\hat{a}}+\frac{1}{\hat{a}}+\frac{1}{\hat{q}}.
     \end{equation*}
     \end{remark}

\begin{proof}[Proof of Lemma \ref{lem Tnlest}]
By the similar approach as the proof of \eqref{est:line:3}, \eqref{eq th2} can be derived by replacing $v$ with $\nabla u$ in \eqref{eq tl2}. Therefore, we focus on proving \eqref{eq tl2}.
As before, we divide the estimation into regions  $B$ and $B^C$. We define an $L^2$-admissible pair
$$r_0=\frac{\alpha+2}{1-b},\qquad q_0=\frac{4\alpha+8}{\alpha+2b}.$$
Selecting $r$, $\gamma$, $q$ and $\theta$ such that $(q,r)\in \mathcal{A}_0$ and
$$\frac{1}{{r_0}'}=\frac{1}{\gamma}+\frac{1}{\infty}+\frac{\alpha+1}{r},~~~~\text{and}~~~~\frac{1}{{q_0}'}=\theta+\frac{1}{\infty}+\frac{1}{q}, $$
by H\"{o}lder's inequality, we have
$$  \||x|^{-b}|u|^\alpha v\|_{L^{{q_0}'}_t(I)L^{{r_0}'}_x(A)}\lesssim\||x|^{-b}\|_{L^\gamma(A)}\|1\|_{L^\frac{1}{\theta}_t(I)L^\infty_x}\|u\|^{\alpha}_{L^\infty_t(I) L^r_x}\|v\|_{L^q_t(I)L^r_x},$$
where $A$ can be either $B$ or $B^C$.

If $A=B$, we choose $r=\frac{\alpha+2}{1-b}+\epsilon$ with a sufficiently small $\epsilon>0$ and the corresponding $q$. Notice that $r>\frac{\alpha+2}{1-b}>2$ and
$$\frac{1}{\gamma}=\frac{1}{{r_0}'}-\frac{\alpha+1}{r}>\frac{\alpha+1+b}{\alpha+2}-\frac{(1-b)(\alpha+1)}{\alpha+2}=b,$$
thus $|x|^{-b}\in L^{\gamma}(B)$.
If $A=B^C$, we choose $r=\alpha+2$ and the corresponding $q$, so that $\frac{1}{\gamma}=\frac{b}{\alpha+2}<b$. Thus $|x|^{-b}\in L^{\gamma}(B^C)$.

In both cases, since $H^1\subset L^r$, we have
$$\||x|^{-b}|u|^\alpha v\|_{L^{{q_0}'}_t(I)L^{{r_0}'}_x(A)}\lesssim T^\theta \|u\|^{\alpha}_{L^\infty_t(I) H^1_x}\|v\|_{L^q_t(I)L^r_x}.$$
On the other hand, for any $(q,r)\in \mathcal{A}_0$, we have $q>4$. This implies $\frac{1}{2}<\theta<\frac{3}{4}$ in any case. Hence there exists $\theta\in(\frac{1}{2},\frac{3}{4})$ such that \eqref{eq tl2} holds.
\end{proof}

In order to obtain the scattering result, we also require the following linear profile decomposition for sequences that are bounded in $H^1$.


\begin{lemma}\label{lem lpdo}
Let $\{\phi_n\}$ be a sequence of bounded odd functions in $H^1(\R)$. Then there exists $M^*\in\N\bigcup\{\infty\}$ so that for each $1\leq j\leq M\leq M^*$, there exists an odd profiles sequence $\Psi^j_n$ in $H^1(\R)$, with parameters $t^j_n\in\mathbb{R}$ and $x^j_n\geq0$, as well as an odd sequence $W^M_n(x)$ of remainders in $H^1(\R)$ such that (up to a subsequence)
\begin{equation*}
        \phi_n(x)=\sum^M_{j=1} e^{-it^j_n\Delta}\Psi^j_n(x)+W^M_n(x).
\end{equation*}
The following properties hold for this decomposition:

(1) The profiles sequence is given by:
        \begin{equation*}
            \Psi^j_n(x)=\psi^j(x-x^j_n)-\psi^j(-x-x^j_n)
        \end{equation*}
with the support of $\psi^j$ contained in $[0,+\infty)$.

(2) Asymptotic orthogonality of the parameters: for $1\leq k\neq j\leq M$,
\begin{equation}\label{eq orth}
            \lim_{n\rightarrow+\infty}|t^j_n-t^k_n|+|x^j_n-x^k_n|=\infty.
\end{equation}

(3) Asymptotic vanishing for the remainders:
\begin{equation}\label{eq vanish rem}\lim_{M\rightarrow+\infty}\left(\lim_{n\rightarrow+\infty}\|e^{it\Delta}W^M_n\|_{S(\dot{H}^{s_c})}\right)=0.
\end{equation}

(4) Asymptotic mass/energy decoupling: for any $s\in[0,1]$, we have \begin{equation}\label{m-e do}
\|\phi_n\|^2_{\dot{H}^s}=\sum^M_{j=1}\|\Psi^j_n\|^2_{\dot{H}^s}+\|W^M_n\|^2_{\dot{H}^s}+o_n(1)\quad as\quad n\rightarrow\infty.
        \end{equation}

(5) Energy Pythagorean expansion:
\begin{equation*}
        E[\phi_n]=\sum^M_{j=1}E[e^{-it^j_n\Delta}\Psi^j_n]+E[W^M_n]+o_n(1).
\end{equation*}

Furthermore, we may assume that either $t^j_n\equiv0$ or $t^j_n\rightarrow\pm\infty$, and either $x^j_n\equiv0$ or $|x^j_n|\rightarrow\infty$.
\end{lemma}
\begin{proof}
Indeed, for any odd function $\phi\in H^1(\R)$, we have
\begin{align*}
    \phi(x)=&\phi(x)|_{(-\infty,0]}+\phi(x)|_{[0,+\infty)}\\
    =&-\phi(-x)|_{[0,+\infty)}+\phi(x)|_{[0,+\infty)}.
\end{align*}
From \cite[Proposition 2.5]{CFGM2022} and \cite[Theorem 5.1]{FXC2011}, we can conclude that Lemma \ref{lem lpdo} is valid with the exception of (5). To prove (5), we have
\begin{equation*}
        E[\phi_n]-\sum^M_{j=1}E[e^{-it^j_n\Delta}\Psi^j_n]-E[W^M_n]=\frac{A_n}{\alpha+2}+o_n(1),
\end{equation*}
where
\begin{equation*}
        A_n\!=\!2\left\||x|^{-b}|\phi_n|^{\alpha+2}\right\|_{L^1(0,+\infty)}\!-\!2\sum^M_{j=1}\left\||x|^{-b}|e^{-it^j_n\Delta}\psi^j(\cdot-x^j_n)|^{\alpha+2}\right\|_{L^1}-\left\||x|^{-b}\left|W^M_n\right|^{\alpha+2}\right\|_{L^1}.
\end{equation*}
The rest of proof exactly proceeds as in \cite[Proposition 5.3]{FG2020},  we omit the details here.
\end{proof}



To complete the analysis, we also need the decay estimate of the equation \eqref{eqinls}. The following estimate can be found in \cite[Theorem 1.4]{CLZ2024}.
\begin{lemma}\label{lem decay}
 Let $u(t)\in H^1(\R)$ be a global odd solution of \eqref{eqinls}.  Then for any bounded interval $I\subset\R$,
 \begin{equation*}
         \lim_{t\rightarrow\infty}(\|u(t)\|_{L^2(I)}+\|u(t)\|_{L^\infty(I)})=0.
\end{equation*}
\end{lemma}
\begin{proof}
    For the reader's convenience, we provide a sketch of proof. In \cite{CLZ2024}, the authors utilize a Virial-Morawetz identity
$$I(u(t)):=\Im\int_{\R}\phi(x)u(t,x)\overline{u}_x(t,x)dx,$$
    by choosing $\phi(x)=\frac{x}{1+|x|}$ to show Lemma \ref{lem decay}. We refer readers to \cite{CLZ2024} for more details.
\end{proof}

\section{Proof of Theorem \ref{T1}}

In this section, our aim is to prove Theorem \ref{T1}. The approach we take is analogous to the method used for the classical NLS equation and the higher-dimensional INLS equation, see \cite{FG2017,FG2020,CFGM2022,Guz2017,Cz2003}.

\begin{proof}[ Proof of Theorem \ref{T1}]
 To prove (1), we first define
 \begin{equation*}
        X:=C([-T,T],H^1(\R))\cap L^q([-T,T],W^{1,r}(\R)),\quad \forall~(q,r)\in\mathcal{A}_0,
\end{equation*}
and
\begin{equation*}
        \|u\|_X:=\|u\|_{S(I,L^2)}+\|\nabla u\|_{S(I,L^2)}.
    \end{equation*}
It is easy to prove that
\begin{equation*}
B_M:=\{u\in X:~\|u\|_X\leq M,~u ~\text{is an odd function}\}
\end{equation*}
equipped with the metric $d(u,v):=\|u-v\|_{S(I,L^2)}$ is a complete metric space.

Next, let us show that
\begin{equation*}
        D(u):=e^{it\Delta}u_0+i\int^t_0e^{i(t-s)\Delta}(-|x|^{-b}|u|^\alpha u)ds
    \end{equation*}
is a contraction map on the metric space $(B_M,d)$ for some $T$, $M>0$. By applying Lemmas \ref{str est} and \ref{lem Tnlest}, we obtain 
\begin{align*}
    \|D(u)\|_X\leq& C\left(\|u_0\|_{H^1}+\||x|^{-b}|u|^\alpha u\|_{S'(I,L^2)}+\|\nabla|x|^{-b}|u|^\alpha u\|_{S'(I,L^2)}\right)\notag\\
        \leq&  C\left(\|u_0\|_{H^1}+T^\theta\|u\|^\alpha_{L^\infty_t(I)H^1_x}\|u\|_X\right).
\end{align*}
Then, $u\in B_M$ implies that
\begin{equation*}
         \|D(u)\|_X\leq C(\|u_0\|_{H^1}+T^\theta M^{\alpha+1}).
\end{equation*}
Next, we set $M=2C\|u_0\|_{H^1}$ and choose $T>0$ such that
\begin{equation*}
        CT^\theta M^\alpha<\frac{1}{4}.
    \end{equation*}
With these choices, we get $D(u)\in B_M$.

By the same argument as above, we can deduce that
\begin{equation*}
        d(D(u),D(v))<\frac{1}{2}d(u,v).
\end{equation*}
 Therefore, by the contraction mapping theorem, there exists a unique fixed point $u\in B_M$. The proof of continuous dependence follows a similar reasoning as above, and thus is omitted here.

To prove (2), we define 
\begin{align*}
        B:=\Big\{u~\text{is an odd function} :\|u\|_{S(\dot{H}^{s_c})}\leq2\left\|e^{it\Delta}u_0\right\|_{S(\dot{H}^{s_c})} \\~\text{and}~\|u\|_{S(L^2)}+\|\nabla u\|_{S(L^2)}\leq2C\|u_0\|_{H^1}\Big\},
\end{align*}
which is equipped with the metric
\begin{equation*}
    d(u,v):=\|u-v\|_{S(L^2)}+\|u-v\|_{S(\dot{H}^{s_c})}.
\end{equation*}
From Lemma \ref{lem nlest}, and using a standard argument (see \cite[Proposition 4.5]{FG2017}), we can establish the  global theory. We omit specific details of this proof for brevity.

To prove (3), we note that since  $u$ is an odd function, a similar argument as presented in \cite[Proposition 1.4]{FG2020} and \cite[Theorem 1.2 (v)]{AT2021} allows us to easily establish (3). We omit the details for brevity.

To prove (4), we start by defining the interval $I_T=[T,+\infty)$ for $T>>1$. We then define 
\begin{equation*}
        G(w)(t):=i\int^{+\infty}_te^{i(t-s)\Delta}(|x|^{-b}|w+e^{it\Delta}\phi|^\alpha(w+e^{it\Delta}\phi))(s)ds,\quad t\in I_T.
\end{equation*}
We also introduce 
\begin{equation*}
        B(T,\rho):=\{w\in C(I_T,H^1(\R)):\|w\|_T\leq\rho,~w~\text{is an odd function}\},
\end{equation*}
where
\begin{equation*}
\|w\|_T=\|w\|_{S(I_T,\dot{H}^{s_c})}+\|w\|_{S(I_T,L^2)}+\|\nabla w\|_{S(I_T,L^2)}.
\end{equation*}
It is important to note that $w+e^{it\Delta}\phi$ is an odd function, therefore, by applying Lemma \ref{lem nlest}, we obtain
\begin{equation*}
       \|G(w)\|_T\leq\|w+e^{it\Delta}\phi\|^\theta_{L^\infty_{I_T}H^1_x}\|w+e^{it\Delta}\phi\|^{\alpha-\theta}_{S(I_T,\dot{H}^{s_c})}\|w+e^{it\Delta}\phi\|_T.
\end{equation*}
The rest of the proof follows a similar approach used in \cite[Proposition 5.3]{FG2017} and \cite[Proposition 4.15]{FG2020}. We  omit the details here. This completes the proof of Theorem \ref{T1}.
\end{proof}

\section{Stability result}

In this section, we present the stability result in the one-dimensional case that is divided in two
parts: short-time perturbation and long-time perturbation. For the results pertaining to higher dimensions, please refer to \cite{FG2017, FG2020, CFGM2022, Dinh2021-3}.
\begin{lemma}[Short-time perturbation]\label{lem short}
Let $I\subset \R$ be a time interval containing zero and  $4-2b<\alpha<\infty$.  Assume that $\tilde{u}=\tilde{u}(t,x)$ is a solution to
\begin{equation*}
        i\partial_t\tilde{u}+\Delta\tilde{u}-|x|^{-b}|\tilde{u}|^\alpha\tilde{u}=e,
\end{equation*}
with initial data $\tilde{u}_0\in H^1(\R)$, where both $\tilde{u}_0$ and $e$ are odd functions. Assume that
    \begin{equation*}
        \sup_{t\in I}\|\tilde{u}(t)\|_{H^1}\leq M~ \text{and}~\|\tilde{u}\|_{S(I,\dot{H}^{s_c})}\leq\eta,
    \end{equation*}
for some $M>0$, and sufficiently small $\eta>0$ depending on $M$. Additionally, let $u_0\in H^1(\R)$ be an odd function that satisfies
\begin{equation*}
        \|u_0-\tilde{u}_0\|_{H^1}\leq M',\quad \text{for}~ M'>0.
\end{equation*}
We also assume that
\begin{equation*}
    \left\|e^{it\Delta}(u_0-\tilde{u}_0)\right\|_{S(I,\dot{H}^{s_c})}\leq\epsilon,
\end{equation*}
and
\begin{equation*}
    \|e\|_{S'(I,L^2)}+\|\nabla e\|_{S'(I,L^2)}+\|e\|_{S'(I,\dot{H}^{-s_c})}\leq\epsilon.
\end{equation*}
Then, there exists a constant $\epsilon_0(M,M')>0$  such that if $0<\epsilon<\epsilon_0$, there exists a unique odd solution  $u$ to equation \eqref{eqinls} on $I\times\R$ with initial data $u_0$ that satisfies
\begin{equation}\label{eqs1}
    \|u-\tilde{u}\|_{S(I,\dot{H}^{s_c})}\lesssim\epsilon
\end{equation}
and
\begin{equation}\label{eqs2}
        \|u\|_{S(I,L^2)}+\|\nabla u\|_{S(I,L^2)}+\|u\|_{S(I,\dot{H}^{s_c})}\lesssim 1.
\end{equation}
\end{lemma}

\begin{proof}
We claim that if $\|\tilde{u}\|_{S(I,\dot{H}^{s_c})}\leq\eta$ for some sufficiently small $\eta>0$ depending on $M$, then
\begin{equation}\label{claim:short}
        \|\tilde{u}\|_{S(I,L^2)}+\|\nabla\tilde{u}\|_{S(I,L^2)}< M.
    \end{equation}
Next, we prove the existence of a solution for the following problem
\begin{equation*}
\left\{\begin{aligned}
			&i\partial_tw+\Delta w-H(x,\tilde{u},w)+e=0,\\
			&w(0,x)=u_0(x)-\tilde{u}_0(x),
		\end{aligned}\right.
\end{equation*}
where $H(x,\tilde{u},w)=|x|^{-b}(|\tilde{u}+w|^\alpha(\tilde{u}+w)-|\tilde{u}|^\alpha\tilde{u})$. Let
\begin{align*}
        D(w)(t):=e^{it\Delta}w_0+i\int^t_0e^{i(t-s)\Delta}(-H+e)(s)ds,
    \end{align*}
and define
\begin{align*}
        B(\rho,K)=\{w\in C(I,H^1(\R)): \|w\|_{S(I,\dot{H}^{s_c})}\leq\rho,\|w\|_{S(I,L^2)}+\|\nabla w\|_{S(I,L^2)}\leq K\}.
\end{align*}
By Lemma \ref{str est}, we have
\begin{align}
            \|D(w)\|_{S(I,L^2)}&\lesssim\|w_0\|_{L^2}+\|H(\cdot,\tilde{u},w)\|_{S'(I,L^2)}+\|e\|_{S'(I,L^2)},\label{eq D1}\\
            \|\nabla D(w)\|_{S(I,L^2)}&\lesssim\|\nabla w_0\|_{L^2}+\|\nabla H(\cdot,\tilde{u},w)\|_{S'(I,L^2)}+\|\nabla e\|_{S'(I,L^2)},\label{eq D2}\\
            \|D(w)\|_{S(I,\dot{H}^{s_c})}&\lesssim\|e^{it\Delta}w_0\|_{S(I,\dot{H}^{s_c})}+\|H(\cdot,\tilde{u},w)\|_{S'(I,\dot{H}^{-s_c})}+\|e\|_{S'(I,\dot{H}^{-s_c})}.\label{eq D3}
    \end{align}
Notice that $|H|\leq|x|^{-b}(|\tilde{u}|^\alpha|w|+|w|^{\alpha+1})$, so
\begin{align*}
    \left\|H(\cdot,\tilde{u},w)\right\|_{S'(I,X)}\leq\left\||x|^{-b}|\tilde{u}|^\alpha w\right\|_{S'(I,X)}+\left\||x|^{-b}|w|^\alpha w\right\|_{S'(I,X)},
\end{align*}
where  $X=L^2 $ or $\dot{H}^{-s_c}$. Using Lemma \ref{lem nlest}, we obtain
\begin{align}\label{eq H13}
     \|H(\cdot,\tilde{u},w)\|_{S'(I,X)}\lesssim\left(\|\tilde{u}\|^\theta_{L^\infty_IH^1_x}\|\tilde{u}\|^{\alpha-\theta}_{S(I,\dot{H}^{s_c})}+\|w\|^\theta_{L^\infty_IH^1_x}\|w\|^{\alpha-\theta}_{S(I,\dot{H}^{s_c})}\right)\|w\|_{S(I,X)}.
\end{align}
Similarly, since
\begin{align*}
|\nabla H| \lesssim|x|^{-b-1} \left(|\tilde{u}|^\alpha \!+\!|w|^\alpha\right)|w| \!+\!|x|^{-b}\left(|\tilde{u}|^\alpha + |w|^\alpha\right)|\nabla w|\!+\!|x|^{-b}\left(|\tilde{u}|^{\alpha-1}\!+\!|w|^{\alpha-1}\right)|w||\nabla\tilde{u}|,
\end{align*}
by Lemma \ref{lem nlest} and \eqref{rem nlest}, we have
\begin{equation}\label{eq H2}
    \begin{aligned}
     &\|\nabla H(\cdot,\tilde{u},w)\|_{S'(I,L^2)}\\
     \lesssim&\left\||x|^{-b}(|\tilde{u}|^\alpha+|w|^\alpha)\frac{|w|}{|x|}\right\|_{S'(I,L^2)}+\left\||x|^{-b}(|\tilde{u}|^\alpha+|w|^\alpha)|\nabla w|\right\|_{S'(I,L^2)}\\
     &+\left\||x|^{-b}|\tilde{u}|^{\alpha-1}|w||\nabla \tilde{u}|\right\|_{S'(I,L^2)}+\left\||x|^{-b}|w|^\alpha|\nabla\tilde{u}|\right\|_{S'(I,L^2)}\\
     \lesssim&\left(\|\tilde{u}\|^\theta_{L^\infty_IH^1_x}\|\tilde{u}\|^{\alpha-\theta}_{S(I,\dot{H}^{s_c})}+\|w\|^\theta_{L^\infty_IH^1_x}\|w\|^{\alpha-\theta}_{S(I,\dot{H}^{s_c})}\right)\|\nabla w\|_{S(I,L^2)}\\
     &+\|w\|^\theta_{L^\infty_IH^1_x}\|w\|^{\alpha-\theta}_{S(I,\dot{H}^{s_c})}\|\nabla\tilde{u}\|_{S(I,L^2)}+\|\tilde{u}\|^\theta_{L^\infty_IH^1_x}\|\tilde{u}\|^{\alpha-1-\theta}_{S(I,\dot{H}^{s_c})}\|w\|_{S(I,\dot{H}^{s_c})}\|\nabla\tilde{u}\|_{S(I,L^2)}.
\end{aligned}\end{equation}
If $w\in B(\rho,K)$, it yields from \eqref{eq D1},\eqref{eq D2},\eqref{eq D3},\eqref{eq H13} and \eqref{eq H2} that
\begin{align*}
\|D(w)\|_{S(I,L^2)}\leq& C\left(M'+\epsilon+\left(M^\theta\eta^{\alpha-\theta}+K^\theta\rho^{\alpha-\theta}\right)K\right),\\
    \|D(w)\|_{S(I,\dot{H}^{s_c})}\leq& C\left(\epsilon+\left(M^\theta\eta^{\alpha-\theta}+K^\theta\rho^{\alpha-\theta}\right)\rho\right),
\end{align*}
and
\begin{align*}
    \|\nabla D(w)\|_{S(I,L^2)}\leq C\left(M'+\epsilon+\left(M^\theta\eta^{\alpha-\theta}+K^\theta\rho^{\alpha-\theta}\right)K
    +\left(M^\theta\eta^{\alpha-1-\theta}+K^\theta\rho^{\alpha-1-\theta}\right)\rho M\right).
\end{align*}
Taking $\rho=2C\epsilon,~K=3CM'$ and choosing $\eta,\epsilon_0$ sufficiently small such that
\begin{align*}
    C\left(M^\theta\eta^{\alpha-\theta}+K^\theta\rho^{\alpha-\theta}\right)
    <\frac{1}{3}~~\text{and}~~C\left(\epsilon+\left(M^\theta\eta^{\alpha-1-\theta}+K^\theta\rho^{\alpha-1-\theta}\right)\rho M\right)<\frac{K}{3},
\end{align*}
these imply that
\begin{align*}
\|D(w)\|_{S(I,\dot{H}^{s_c})}\leq\rho~~\text{and}~~\|D(w)\|_{S(I,L^2)}+\|\nabla D(w)\|_{S(I,L^2)}\leq K.
\end{align*}
Thus, $D$ is well defined on $B(\rho,K)$. The contraction property can be obtained by a similar argument.  Consequently, there is a unique solution $w$ on $I\times\R$ satisfying
\begin{align*}
    \|w\|_{S(I,\dot{H}^{s_c})}\lesssim\epsilon~~\text{and}~~\|w\|_{S(I,L^2)}+\|\nabla w\|_{S(I,L^2)}\lesssim M'.
\end{align*}
In particular, $u=\tilde{u}+w$ is a solution to \eqref{eqinls} such that \eqref{eqs1} and \eqref{eqs2} hold.

Now let us show that claim \eqref{claim:short} holds. Indeed, from Lemmas  \ref{str est} and \ref{lem nlest},  we have
 \begin{align*} \|\nabla\tilde{u}\|_{S(I,L^2)}\lesssim&\left\|\nabla\tilde{u}_0\right\|_{L^2}+\left\|\nabla\left(|x|^{-b}|\tilde{u}|^\alpha\tilde{u}\right)\right\|_{S'(I,L^2)}+\|\nabla e\|_{S'(I,L^2)}\\
 \leq& C\left( M+\epsilon+M^\theta\eta^{\alpha-\theta}\|\nabla\tilde{u}\|_{S(I,L^2)}\right).
\end{align*}
Taking $\eta$ small enough such that $CM^\theta\eta^{\alpha-\theta}<\frac{1}{2}$, it concludes that
\begin{align*}
    \|\nabla\tilde{u}\|_{S(I,L^2)}\lesssim M.
\end{align*}
Similarly, we can also get $\|\tilde{u}\|_{S(I,L^2)}\lesssim M$, which implies the claim. This completes the proof.
\end{proof}

Next, we establish a long-time perturbation result for equation \eqref{eqinls} removing the assumption that $\|\tilde{u}\|_{S(I,\dot{H}^{s_c})}$ is small as in the previous lemma.
\begin{lemma}[Long-time perturbation]\label{prop sta}
Let $I\subset \R$ be a time interval containing zero and $4-2b<\alpha<\infty$. Assume that $\tilde{u}$, defined on $I\times\R$, is a solution to 
\begin{equation*}
        i\partial_t\tilde{u}+\Delta\tilde{u}-|x|^{-b}|\tilde{u}|^\alpha\tilde{u}=e,
\end{equation*}
with initial data $\tilde{u}_0\in H^1(\R)$, where both $\tilde{u}_0$ and $e$ are odd functions, and there exist constants $M$, $L>0$ such that
\begin{align*}
        \sup_{t\in I}\|\tilde{u}(t)\|_{H^1}\leq M~~ \text{and}~~\|\tilde{u}\|_{S(I,\dot{H}^{s_c})}\leq L.
\end{align*}
Let $u_0\in H^1(\R)$ be an odd function that satisfies
\begin{align*}
        \|u_0-\tilde{u}_0\|_{H^1}\leq M'
\end{align*}
for some $M'>0$. Additionally, assume that
\begin{align*}
        \|e^{it\Delta}(u_0-\tilde{u}_0)\|_{S(I,\dot{H}^{s_c})}\leq\epsilon,
\end{align*}
and
\begin{align*}
        \|e\|_{S'(I,L^2)}+\|\nabla e\|_{S'(I,L^2)}+\|e\|_{S'(I,\dot{H}^{-s_c})}\leq\epsilon.
\end{align*}
Under these conditions, there exists  $\epsilon_1(M,M',L)>0$ such that if $0<\epsilon<\epsilon_1$, then there exists a unique odd solution $u$ to \eqref{eqinls} on $I\times\R$ with initial data $u_0$. This solution satisfies the following estimates
\begin{align*}
        \|u-\tilde{u}\|_{S(I,\dot{H}^{s_c})}\lesssim_{M,M',L}\epsilon
\end{align*}
and
\begin{align*}
        \|u\|_{S(I,L^2)}+\|\nabla u\|_{S(I,L^2)}+\|u\|_{S(I,\dot{H}^{s_c})}\lesssim_{M,M',L} 1.
    \end{align*}
\end{lemma}
\begin{proof}
 By combining Lemma \ref{lem short} with a similar argument as presented in \cite[Proposition 4.14]{FG2020}, we can derive the desired result. Here we omit the details.
\end{proof}

\section{Existence and compactness of a critical element}

In this section, we will prove that if Theorem \ref{T2} does not hold, then one can find a minimal blow-up solution that obeys the compactness property. We will follow the approach outlined in \cite{CFGM2022, KM2006} to complete the proof.
Before starting, we introduce the following quantities that possess a scaling-invariant property
$$M[u_\lambda]^{1-s_c}E[u_\lambda]^{s_c}=M[u]^{1-s_c}E[u]^{s_c},~~~~\|\nabla u_\lambda\|_{L^2_x}^{s_c}\|u_\lambda\|_{L^2_x}^{1-s_c}=\|\nabla u\|_{L^2_x}^{s_c}\|u\|_{L^2_x}^{1-s_c}.$$
These quantities were first introduced in \cite{HR2008}, also see \cite{DHR2008}, in their study of the cubic NLS equation in three dimensions. They are  essential in distinguishing between blow-up and global solutions. In our analysis, these quantities also play a crucial role.

\begin{prop}\label{prop c}
Suppose that Theorem \ref{T2} does not hold. Then there exists a critical threshold $\delta_c>0$ and an odd function $u_{c,0}\in H^1(\R)$ such that $u_c$ is a solution of \eqref{eqinls} with initial data $u_{c,0}$, uniformly bounded in $H^1(\R)$ and satisfies the following properties

(1) $M[u_c]=1$;

(2) $E[u_c]^{s_c}=\delta_c$;

(3) $\|u_c\|_{S(\dot{H}^{s_c})}=+\infty$.

\noindent Furthermore, the orbit $\{u_c(t):~t\in\R\}$ is pre-compact in $H^1(\R)$.
\end{prop}
\begin{proof} 
The idea of this proof can be found, for example, in \cite{CFGM2022,Visan2014} in the context of focusing INLS or defocusing energy-critical NLS in high dimensions. For the reader’s convenience, we provide the details of the adaptation to our case. Firstly, we show the existence of the critical threshold $\delta_c$. We define
\begin{equation}\label{eq L}
        L(\delta):=\sup\left\{\|u\|_{S(\dot{H}^{s_c})}:u~\text{is an odd solution to \eqref{eqinls} such that}~ M[u]^{1-s_c}E[u]^{s_c}\leq\delta\right\}.
\end{equation}
According to Theorem \ref{T1} (2), there exists a unique critical threshold $\delta_c\in(0,\infty]$ such that $L(\delta)\!<\!\infty$ for $\delta\!<\!\delta_c$, and $L(\delta)\!=\!\infty$ for $\delta\geq\delta_c$.  If Theorem \ref{T2} fails, it follows that $0\!<\!\delta_c\!<\!\infty$.

On the other hand, according to \cite{Farah2016}, for any $u_0\in H^1(\R)$, the corresponding solution $u$ to \eqref{eqinls} is always global and uniformly bounded in $H^1(\R)$.  Therefore, we can take a sequence of odd solutions $u_n$ to \eqref{eqinls}, each with odd initial data $u_{n,0}\in H^1(\R)$ that satisfy
\begin{align}
&\|u_{n,0}\|_{L^2}=1, \label{eq am}\\
&E[u_n]^{s_c}\rightarrow\delta_c~~\text{as}~~n\rightarrow\infty, \label{eq ae} \\
&\lim_{n\rightarrow\infty}\|u_n\|_{S(\R^+,\dot{H}^{s_c})}=\lim_{n\rightarrow\infty}\|u_n\|_{S(\R^-,\dot{H}^{s_c})}=\infty. \label{eq abup}
\end{align}

From Lemma \ref{lem lpdo}, we have
\begin{equation}\label{und}
        u_{n,0}(x)=\sum^M_{j=1}e^{-it^j_n\Delta}\Psi^j_n(x)+W^M_n(x),
    \end{equation}
where $\Psi^j_n(x)=\psi^j(x-x^j_n)-\psi^j(-x-x^j_n)$ and the support of $\psi^j$ is contained in $[0,+\infty)$. We will now prove the following assertions:
\begin{itemize}
        \item[(a)] There is a single profile $\Psi^1_n(x)=\psi^1(x-x^1_n)-\psi^1(-x-x^1_n)$ in the decomposition \eqref{und}, i.e. $\Psi^j_n,\psi^j\equiv0$ for $j\geq2$.
        \item[(b)] The time-space translation parameters $(t_n,x_n)\equiv(0,0)$.
        \item[(c)] The remainders $W_n\rightarrow0$ in $H^1(\R)$.
\end{itemize}

If (a)-(c) hold, then the sequence $u_{n,0}$ has a strong convergence limit $u_{c,0}$, up to a subsequence.  By strong convergence and Lemma \ref{prop sta},  $u_{c,0}$ and the corresponding solution $u_c$ satisfy properties (1)-(3).

We now turn to prove (a)-(c) through three steps.

\textbf{Step 1.} To prove (a), we proceed by contradiction, assuming that more than one profile appears in the decomposition \eqref{und}. It is important to note that there must be at least one profile; otherwise, by Lemma \ref{prop sta}, the approximate solution $e^{it\Delta}W^M_n$ would lead to global space-time bounds for the solutions $u_n$, which would contradict \ref{eq abup}. According to Lemma \ref{lem lpdo}, along with \eqref{eq am} and \eqref{eq ae}, we can deduce that each profile satisfies 
\begin{align*}
    M[\Psi^j_n]^{1-s_c}E[e^{-it^j_n}\Psi^j_n]^{s_c}<\delta_c~~\text{for all $j$ and all large $n$}.
\end{align*}
This allows us to construct scattering solutions $v^j_n$ to \eqref{eqinls} corresponding to each profile. The function $v^j_n$ is constructed according to three different cases.

{\bf Case 1.} If $x^j_n\equiv0$ and $t^j_n\equiv0$, then the solution $v^j$ to \eqref{eqinls} with initial data $\Psi^j$ scatters, as it lies below the critical threshold $\delta_c$.

{\bf Case 2.} If $x^j_n\equiv0$ and $t^j_n\rightarrow\pm\infty$, we define $v^j$  as the solution to \eqref{eqinls} that scatters to $e^{it\Delta}\Psi^j$ as $t\rightarrow\pm\infty$.
The existence of $v^j$ is ensured by the existence of the wave operator stated in conclusion (4) of Theorem \ref{T1}.

In Case 1 and  2, since $x^j_n\equiv0$, meaning that $\Psi^j_n$ is independent of $n$, we can replace $\Psi^j_n$ with $\Psi^j$. Consequently, we define $v^j_n:=v^j(t+t^j_n,x)$.

{\bf Case 3.} If $|x^j_n|\rightarrow+\infty$, we construct a global scattering solution $v^j_n$ to \eqref{eqinls} with initial data given by $v^j_n(0,x)=e^{-it^j_n\Delta}\Psi^j_n$. The global and scattering properties of this solution will be established in Proposition \ref{prop far}.

We define a sequence as following:
\begin{align*}
        \tilde{u}^M_n:=\sum^M_{j=1} v^j_n,
\end{align*}
which serves as an approximate solution for $u_n$ to the equation \eqref{eqinls}.

We give the following claims:

{\bf Claim 1.} Asymptotic close between initial data:
 \begin{align*}
            \lim_{M\rightarrow\infty}\left[\lim_{n\rightarrow\infty}\|e^{it\Delta}(u_n(0)-\tilde{u}^M_n(0))\|_{S(\dot{H}^{s_c})}\right]=0.
        \end{align*}

{\bf Claim 2.} $\tilde{u}^M_n$ are approximate solutions to \eqref{eqinls}:   $\tilde{u}^M_n$ satisfy
\begin{equation*}
    (i\partial_t+\Delta)\tilde{u}^M_n-|x|^{-b}|\tilde{u}^M_n|^{\alpha}\tilde{u}^M_n=e^M_n,
\end{equation*}
where $e^M_n=|x|^{-b}(\sum^M_{j=1}(|v^j_n|^\alpha v^j_n)-|\sum^M_{j=1}v^j_n|^{\alpha}\sum^M_{j=1}v^j_n)$. Moreover for each $M$ and $\epsilon>0$, there exists $n_0=n_0(M,\epsilon)$ such that for any $n>n_0$,
\begin{equation}\label{eq small}
\|e^M_n\|_{S'(\dot{H}^{-s_c})}+\|e^M_n\|_{S'(L^2)}+\|\nabla e^M_n\|_{S'(L^2)}\leq\epsilon.
\end{equation}

{\bf Claim 3.} Uniform time-space bounds for $\tilde{u}^M_n$: There exist constants $L$, $S>0$ that are independent of $M$, and for each $M$ there exists $n_1=n_1(M)$ such that for any $n>n_1$,
\begin{align*}
    \|\tilde{u}^M_n\|_{S(\dot{H}^{s_c})}\leq L~\text{and}~\|\tilde{u}^M_n\|_{L^\infty_tH^1_x}\leq S.
\end{align*}

If these three claims are valid, we can conclude that the solutions $u_n$  possess time-space bounds. This is because, according to Lemma \ref{prop sta}, the solutions $u_n$  inherit the bounds of the approximate solutions $\tilde{u}^M_n$ when both $M$ and $n$ are sufficiently large. This leads to a contradiction to \eqref{eq abup}. Therefore, assertion $(a)$ can be established.

We now provide the proofs for the three claims.

{\bf Proof of Claim 1.} We begin with the expression for the difference between the initial data:
\begin{align*}
          u_n(0,x)-\tilde{u}^M_n(0,x)=\sum^M_{j=1}\left[e^{-it^j_n\Delta}\Psi^j_n(x)-v^j_n(0,x)\right]+W^M_n.
      \end{align*}
Observe that each term in the sum is either zero or converges to zero in $H^1(\R)$. Applying the interpolation inequality and
and \eqref{eq vanish rem}, we can conclude Claim 1.

{\bf Proof of Claim 2.} By the construction, $\tilde{u}^M_n$ satisfies the following approximate equation
\begin{equation*}
    (i\partial_t+\Delta)\tilde{u}^M_n-|x|^{-b}|\tilde{u}^M_n|^{\alpha}\tilde{u}^M_n=e^M_n.
\end{equation*}
Now, let's verify the condition \eqref{eq small}. By using some elementary inequalities, we have
\begin{align*}
    |e^M_n|\lesssim|x|^{-b}\left|\sum_{1<j\neq k<M}|v^k_n|^\alpha|v^j_n|\right|,
\end{align*}
and for the gradient,
\begin{align*}
    |\nabla e^M_n|\lesssim|x|^{-b-1}\left|\sum_{1\leq j\neq k\leq M}|v^k_n|^\alpha|v^j_n|\right|+|x|^{-b}\left|\sum_{1\leq j\neq k\leq M}|v^k_n|(|v^k_n|^{\alpha-1}+|v^j_n|^{\alpha-1})|\nabla v^j_n|\right|.
\end{align*}
These imply:
\begin{align*}
    \left\|e^M_n\right\|_{S'(\dot{H}^{-s_c})}\lesssim&\sum_{1\leq j\neq k\leq M}\left\||x|^{-b}|v^k_n|^\alpha|v^j_n|\right\|_{L^{\tilde{a}'}_tL^{\hat{r}'}_x},\\
   \left\|e^M_n\right\|_{S'(L^2)}\lesssim&\sum_{1\leq j\neq k\leq M}\left\||x|^{-b}|v^k_n|^\alpha|v^j_n|\right\|_{L^{\hat{q}'}_tL^{\hat{r}'}_x},\\
   \left\|\nabla e^M_n\right\|_{S'(L^2)}\lesssim&\sum_{1\leq j\neq k\leq M}\left\||x|^{-b-1}|v^k_n|^\alpha|v^j_n|\right\|_{L^{\hat{q}'}_tL^{\hat{r}'}_x}\\
   &+\sum_{1\leq j\neq k\leq M}\left\||x|^{-b}|v^k_n|(|v^k_n|^{\alpha-1}+|v^j_n|^{\alpha-1})|\nabla v^j_n|\right\|_{L^{\hat{q}'}_tL^{\hat{r}'}_x}.
\end{align*}
Here $\tilde{a},\hat{q},\hat{r}$ are the same parameters as in the proof of Lemma \ref{lem nlest}.

By the construction of $v^j_n$ outlined in Cases 1-3, each $v^j_n\in S(\dot{H}^{s_c})$ can be approximated by functions in $C^\infty_c(\R\times\R)$. Utilizing the orthogonality of $v^j_n$ presented by  \eqref{eq orth}, it follows that all summands converge to zero as $n\rightarrow\infty$. Thus, \eqref{eq small} holds, confirming Claim 2.

{\bf Proof of Claim 3.} 
By the definition of the energy $E$ (see \eqref{E}), we derive:
\begin{align*}
  \limsup_{n\rightarrow\infty}\sum^M_{j=1}\|v^j_n\|^2_{L^\infty_tH^1_x}\lesssim\limsup_{n\rightarrow\infty}
  \sum^M_{j=1}\left(M[v^j_n]+E[v^j_n]\right)\lesssim 1.
\end{align*}
The orthogonality condition \eqref{eq orth} implies that for $j\neq k$
\begin{align*}
    \sup_{t\in\R}\left|(v^j_n,v^k_n)_{H^1}\right|\rightarrow0\quad \text{as}\quad n\rightarrow\infty,
\end{align*}
as shown in \cite[Corollary 4.4]{FXC2011}. With the construction of $\tilde{u}^M_n$, there exists a constant $S>0$, independent of $M$, such that
\begin{align}\label{eq S}
    \sup_{t\in\R}\|\tilde{u}^M_n\|^2_{H^1}\leq S\quad\text{for all}\quad n>n_1(M).
\end{align}
Next, we establish the bound for the time-space norm of $\tilde{u}^M_n$. Consider
\begin{align*}
\|\tilde{u}^M_n\|^2_{L^{\hat{a}}_tL^{\hat{r}}_x}\leq\sum^{M_0}_{j=1}\|v^j_n\|^2_{L^{\hat{a}}_tL^{\hat{r}}_x}
+\sum^{M}_{j=M_0+1}\|v^j_n\|^2_{L^{\hat{a}}_tL^{\hat{r}}_x}+\sum_{j\neq k}\|v^j_nv^k_n\|_{L^{\frac{\hat{a}}{2}}_tL^{\frac{\hat{r}}{2}}_x},
\end{align*}
where $(\hat{a},\hat{r})\in \mathcal{A}_{s_c}$ is given in Lemma \ref{lem nlest}.

Using \eqref{eq am}, \eqref{eq ae}, and \eqref{m-e do}, we find that there exists $C>0$ such that
\begin{align*}
    \sum^\infty_{j=1}\|\Psi^j_n\|^2_{H^1}\leq C,
\end{align*}
for sufficiently large $n$. This implies that for some small $\delta\in(0,1)$ and sufficiently large $n$, there exists $M_0(\delta)>0$ such that
\begin{align*}
\sum^\infty_{j=M_0(\delta)}\|\Psi^j_n\|^2_{H^1}\leq \delta.
\end{align*}
From Lemma \ref{str est} and Theorem \ref{T1}(2), we derive
\begin{align*}
\sum^M_{j=M_0(\delta)}\|v^j_n\|^2_{L^{\hat{a}}_tL^{\hat{r}}_x}\leq 4\delta \quad\text{for} \quad n>n_1(M).
\end{align*}
The orthogonality condition \eqref{eq orth} ensures
\begin{align*}
\|\tilde{u}^M_n\|^2_{L^{\hat{a}}_tL^{\hat{r}}_x}\leq M_0(\delta)L_1+4\delta+o(1)\leq L^2_2\quad \text{for}\quad n>n_1(M),
\end{align*}
where $L_1$, $L_2>0$ are independent of $M$ (the construction of $v^j_n$ and \eqref{eq L} confirm the existence of $L_1$).

By applying Lemma \ref{lem nlest} and \eqref{eq S}, it follows that
\begin{align*}
    \left\||x|^{-b}|\tilde{u}^M_n|^\alpha\tilde{u}^M_n\right\|_{L^{\tilde{a}'}_tL^{\hat{r}'}_x}\lesssim S^\theta L_2^{\alpha+1-\theta}.
\end{align*}
Moreover, using \eqref{eq small}, \eqref{eq S}, and Lemma \ref{str est}, we have,  for  $n>n_1(M)$, the following inequality holds:
\begin{align*}
    \|\tilde{u}^M_n\|_{S(\dot{H}^{s_c})}&\leq C\|\tilde{u}^M_n(0)\|_{H^1}+C\left\||x|^{-b}|\tilde{u}^M_n|^\alpha\tilde{u}^M_n\right\|_{L^{\tilde{a}'}_tL^{\hat{r}'}_x}+\|e^M_n\|_{S'(\dot{H}^{-s_c})}\\
    &\leq CS+CS^\theta L^{\alpha+1-\theta}_2+\epsilon\leq L.
\end{align*}
Thus, we complete the proof of Claim 3.

\textbf{Step 2.} To prove (b), we start by considering the decomposition provided in (a) and the expression \eqref{und}. We rewrite the initial data as
$$u_{n,0}(x)=e^{-it_n\Delta}\Psi_n(x)+W_n(x).$$
If $|x_n|\rightarrow\infty$, the solutions $v_n$ to \eqref{eqinls} with initial data $v_n(0,x)=e^{-it_n\Delta}\Psi_n(x)$ are global and scattering, as given by Proposition \ref{prop far}.  According to Lemma \ref{prop sta}, this allows us to derive uniform time-space bounds for the solutions $u_n$. However, this conclusion contradicts \eqref{eq abup}, which suggests unboundedness norm.
If $|x_n|\equiv0$, $|t_n|\rightarrow\infty$, we use the dispersive estimate to deduce that $e^{-it_n\Delta}\Psi(x)\rightharpoonup0$ in $H^1(\R)$. This weak convergence indicates that the functions $v_n(t)=e^{i(t-t_n)\Delta}\Psi$ are well-behaved approximate solutions for large $n$. Again, applying Lemma \ref{prop sta}, we can obtain uniform time-space bounds for $u_n$, which once more contradicts \eqref{eq abup}. In both cases, we arrive at contradiction with \eqref{eq abup}. This contradiction confirms that (b) holds.

\textbf{Step 3.} To prove (c), it is derived from \eqref{eq am}, \eqref{eq ae}, \eqref{eq abup} and (a) that there is a single profile $\Psi$ satisfying $M[\Psi]=1$ and $E[\Psi]^{s_c}=\delta_c$.  Consequently, we have $W_n\rightarrow0$ in $H^1(\R)$. This means that (c) holds.

Finally, we show that the orbit $\{u_c(t):~t\in\R\}$ is pre-compact in $H^1(\R)$. We replace $u_{n,0}$ with $\phi_n:=u_c(t_n)$, where $\{t_n\}\subseteq(-\infty,+\infty)$ is a sequence converging to some $t^*\in [-\infty,+\infty]$ (up to a subsequence), so that $\phi_n$ still satisfy \eqref{eq am}, \eqref{eq ae} and \eqref{eq abup}. By repeating the proof steps above, we find that $\{\phi_n\}$ has a subsequence that converges strongly in $H^1$. This convergence holds even when $t^*=\infty$.

Thus, we complete the proof of Proposition \ref{prop c}.
\end{proof}

It remains to prove that the solutions to \eqref{eqinls} with initial profiles living far from the origin are global in time and scattering. This approach builds on the method introduced in \cite{MMZ2021} and \cite{CFGM2022}. Our goal is to extend their results to the case where  $N=1$, with the oddness condition.

\begin{prop}\label{prop far}
Suppose $0<b<1 $ and $4-2b<\alpha<\infty$. Let $\psi\in H^1(\R)$ with the support contained in $[0,+\infty)$, and define $\Psi_n(x)=\psi(x-x_n)-\psi(-x-x_n)$. Assume $t_n\equiv0$ or $t_n\rightarrow\pm\infty$ and $x_n\rightarrow+\infty$.  Then, for all sufficiently large $n$, there exist global solutions $v_n$ to \eqref{eqinls} with initial data
\begin{align*}
        v_n(0)=e^{-it_n\Delta}\Psi_n
\end{align*}
that scatter in $H^1(\R)$ and satisfy
\begin{align*}
        \|v_n\|_{S(\dot{H}^{s_c})}+\|v_n\|_{S(L^2)}+\|\nabla v_n\|_{S(L^2)}\leq C,
    \end{align*}
for some constant $C=C(\|\psi\|_{H^1})$.  Furthermore, for any $\epsilon>0$, there exist $K>0$ and $\Phi_n(t,x)$  such that
\begin{align*}
        \|v_n-\Phi_n\|_{S(\dot{H}^{s_c})}<\epsilon\quad\text{for}\quad n>K,
    \end{align*}
where $\Phi_n(t,x):=\phi(t-t_n,x-x_n)-\phi(t-t_n,-x-x_n)$ and $\phi\in C_c^\infty(\R\times\R)$.
\end{prop}
\begin{proof}
 For each $n$, define a smooth cut-off function $\chi_n$ satisfying
\begin{align*}
        \chi_n(x)=\left\{\begin{aligned}
           &1,\quad\text{for}\quad|x+x_n|>\frac{|x_n|}{2},\\
        &0,\quad\text{for}\quad|x+x_n|<\frac{|x_n|}{4},
        \end{aligned}
         \right.
    \end{align*}
with $\sup_x|\nabla^\alpha\chi_n(x)|\lesssim|x_n|^{-\alpha}$ for all $\alpha\in\mathbb{N}$. Note that $\chi_n(x)\rightarrow1$ as $n\rightarrow\infty$ for each $x\in\R$ since $x_n\to+\infty$.

Now let us construct approximate solutions to \eqref{eqinls}. For $T>0$, we define
\begin{align*}
        \tilde{v}_{n,T}(t,x)=\left\{\begin{aligned}
            &\chi_n(x\!-\!x_n)e^{it\Delta}P_n\psi(x\!-\!x_n)\!-\!\chi_n(-x\!-\!x_n)e^{it\Delta}P_n\psi(-x\!-\!x_n),~&for~|t|\leq T,\\
            &e^{i(t-T)\Delta}[\tilde{v}_{n,T}(T)], &for~ t>T,\\
            &e^{i(t+T)\Delta}[\tilde{v}_{n,T}(-T)], &for~ t<-T,
        \end{aligned}\right.
\end{align*}
where $P_n:=P_{\leq|x_n|^\theta}$ are the Littlewood-Paley projections for some small $0<\theta<1$. We will show that $\tilde{v}_{n,T}$ are approximate solutions of $v_n$ by Lemma \ref{prop sta}.  We need the following estimates:
\begin{itemize}

\item[$i$).] $\lim_{T\rightarrow\infty}\limsup_{n\rightarrow\infty}\|\tilde{v}_{n,T}(-t_n)-v_n(0)\|_{H^1}=0$.

\item[$ii$).] $\limsup_{T\rightarrow\infty}\limsup_{n\rightarrow\infty}\left(\|\tilde{v}_{n,T}\|_{L^\infty_tH^1_x}
+\|\tilde{v}_{n,T}\|_{S(\dot{H}^{s_c})}\right)\lesssim 1.$

\item[$iii$).] $\lim_{T\rightarrow\infty}\limsup_{n\rightarrow\infty}\left(\|e_{n,T}\|_{S'(\dot{H}^{-s_c})}+\|e_{n,T}\|_{S'(L^2)}+\|\nabla e_{n,T}\|_{S'(L^2)}\right)=0,$\\
where $e_{n,T}=(i\partial_t+\Delta)\tilde{v}_{n,T}-|x|^{-b}|\tilde{v}_{n,T}|^\alpha \tilde{v}_{n,T}.$
\end{itemize}

To establish these estimates, we will frequently rely on several key facts: the unitarity of $e^{i\cdot\Delta}$ as an operator on $L^2$, the commutation property between the operators $\nabla$ and $e^{i\cdot\Delta}$, as well as  the embedding $H^1\subset\dot{H}^{s_c}$. For the sake of brevity, we may occasionally skip explicitly mentioning these points in the proofs. In what follows, let us proceed to prove the three assertions.

{\textbf{Proof of $i$)}.} If $t_n\equiv0$, we have
\begin{align*}
\begin{split}
            \tilde{v}_{n,T}(-t_n)-v_n(0)=&(\chi_n(x-x_n)P_n\psi(x-x_n)-\psi(x-x_n))\\
           &-(\chi_n(-x-x_n)P_n\psi(-x-x_n)-\psi(-x-x_n)).
\end{split}
\end{align*}
Using the dominated convergence theorem, it follows that
\begin{align*}
            \|\tilde{v}_{n,T}(-t_n)-v_n(0)\|_{H^1}\leq 2\|\chi_nP_n\psi-\psi\|_{H^1}\rightarrow0~~\text{as}~~n\rightarrow\infty,
        \end{align*}
since $P_n\psi\rightarrow\psi$ strongly in $H^1(\R)$.

Similarly, if $t_n\rightarrow\infty$, we can obtain
\begin{align*}\begin{split}
            \|\tilde{v}_{n,T}(-t_n)-v_n(0)\|_{H^1}&\leq2\|e^{-iT\Delta}\chi_ne^{iT\Delta}P_n\psi-\psi\|_{H^1}=\|\chi_ne^{iT\Delta}P_n\psi-e^{iT\Delta}\psi\|_{H^1}\\
            &\leq2(\|(1-\chi_n)e^{iT\Delta}P_n\psi\|_{H^1}+\|P_n\psi-\psi\|_{H^1})\rightarrow0\quad\text{as}\quad n\rightarrow\infty.
        \end{split}
        \end{align*}
Therefore, estimate $i$) holds true.

{\textbf{Proof of $ii$)}.} First, we consider the case $|t|\leq T$. For sufficiently large $n$, we have
\begin{align*}\begin{split}
            \|\tilde{v}_{n,T}\|_{L^\infty_{(|t|\leq T)}H^1_x}\leq&2\|e^{it\Delta}P_n\psi\|_{L^\infty_tL^2_x}+2\|\nabla\chi_n\|_{L^2}\|e^{it\Delta}P_n\psi\|_{L^\infty_tL^\infty_x}\\
            &+2\|\chi_n\|_{L^\infty}\|\nabla e^{it\Delta}P_n\psi\|_{L^\infty_tL^2_x}\\
            \leq& (2+o(1))\|e^{it\Delta}P_n\psi\|_{L^\infty_tH^1_x}\lesssim\|\psi\|_{H^1}.
            \end{split}
\end{align*}
On the other hand, by Lemma \ref{str est}, we obtain
\begin{align*}
        \|\tilde{v}_{n,T}\|_{S((|t|\leq T),\dot{H}^{s_c})}\lesssim\|e^{it\Delta}P_n\psi\|_{S(\dot{H}^{s_c})}\lesssim\|P_n\psi\|_{\dot{H}^{s_c}}\lesssim\|\psi\|_{H^1}.
    \end{align*}
For the case $|t|>T$, we have
\begin{align*}
            \|\tilde{v}_{n,T}\|_{L^\infty_{(|t|> T)}H^1_x}=\|e^{i(t\mp T)\Delta}[\tilde{v}_{n,T}(\pm T)]\|_{L^\infty_tH^1_x}\leq\|\tilde{v}_{n,T}(\pm T)\|_{H^1}\lesssim\|\psi\|_{H^1},
\end{align*}
and
\begin{align*}
            \|\tilde{v}_{n,T}\|_{S((|t|>T),\dot{H}^{s_c})}\lesssim\|e^{\mp iT\Delta}[\tilde{v}_{n,T}(\pm T)]\|_{\dot{H}^{s_c}}\lesssim\|\tilde{v}_{n,T}(\pm T)\|_{H^1}\lesssim\|\psi\|_{H^1}.
\end{align*}
Combining these estimates, we get $ii$).

{\textbf{Proof of $iii$)}.} If $|t|\leq T$, we have
\begin{align*}
\begin{aligned}
     e_{n,T}=&\left[\Delta[\chi_n(x-x_n)]e^{it\Delta}P_n\psi(x-x_n)-\Delta[\chi_n(-x-x_n)]e^{it\Delta}P_n\psi(-x-x_n)\right]\\
        &+2\left[\nabla[\chi_n(x-x_n)]\nabla e^{it\Delta}P_n\psi(x-x_n)-\nabla[\chi_n(-x-x_n)]\nabla e^{it\Delta}P_n\psi(-x-x_n)\right]\\
         &-|x|^{-b}|\tilde{v}_{n,T}|^\alpha\tilde{v}_{n,T}\\
        :=&E_1+E_2+E_3.
 \end{aligned}
\end{align*}
By simple calculation and using H\"{o}lder inequality,  we obtain
\begin{align*}
            \|E_1\|_{L^1_{(|t|\leq T)}L^2_x}\lesssim T\|\Delta\chi_n\|_{L^\infty}\|e^{it\Delta}P_n\psi\|_{L^\infty_tL^2_x}\lesssim T|x_n|^{-2}\|\psi\|_{L^2}\rightarrow0~~\text{as}~~n\rightarrow\infty,
\end{align*}
and
\begin{align*}
            \|E_2\|_{L^1_{(|t|\leq T)}L^2_x}\lesssim T\|\nabla\chi_n\|_{L^\infty}\|\nabla e^{it\Delta}P_n\psi\|_{L^\infty_tL^2_x}\lesssim T|x_n|^{-1}\|\nabla\psi\|_{L^2}\rightarrow0~~\text{as}~~n\rightarrow\infty.
\end{align*}
It yields from the Bernstein estimate \eqref{eq Bern} that
\begin{align*}
    \begin{split}
            \|\nabla(E_1+E_2)\|_{L^1_{(|t|\leq T)}L^2_x}\lesssim& T(|x_n|^{-3}\|\psi\|_{L^2}+|x_n|^{-2}\|\nabla\psi\|_{L^2})+T|x_n|^{-1}\|\Delta P_n\psi\|_{L^2}\\
            \lesssim&T(|x_n|^{-3}+|x_n|^{-2}+|x_n|^{-1+\theta})\|\psi\|_{H^1}\rightarrow0~~\text{as}~~n\rightarrow\infty.
        \end{split}
\end{align*}
For $\|E_1+E_2\|_{S'((|t|\leq T),\dot{H}^{-s_c})}$, we 
note that for any positive integer $\alpha$, the support of $\nabla^\alpha\chi_n$ are contained in $\{\frac{|x_n|}{4}\leq|x+x_n|\leq\frac{|x_n|}{2}\}$. Therefore, as $n\rightarrow\infty$,
\begin{align*}
\begin{split}
            \|E_1+E_2\|_{L^{q'}_{(|t|\leq T)}L^{r'}_x}\lesssim& T^{\frac{1}{q'}}(\|\Delta\chi_n\|_{L^{\frac{2r}{r-2}}}+\|\nabla\chi_n\|_{L^{\frac{2r}{r-2}}})\|e^{it\Delta}P_n\psi\|_{L^\infty_tH^1_x}\\
            \lesssim&T^{\frac{1}{q'}}|x_n|^{-1+\frac{1}{2}-\frac{1}{r}}\|\psi\|_{H^1}\rightarrow0,
\end{split}
\end{align*}
where $(q,r)\in \mathcal{A}_{-s_c}$.

For the term $E_3$, as $n\rightarrow\infty$,  we have
\begin{align*}\begin{split}
             \|E_3\|_{L^1_{(|t|\leq T)}L^2_x}\lesssim& T\||x|^{-b}\|_{L^\infty_{(|x|>\frac{|x_n|}{4})}}\|e^{it\Delta}P_n\psi\|^{\alpha+1}_{L^\infty_tL^{2(\alpha+1)}_x}\\
             \lesssim& T|x_n|^{-b}\|\psi\|^{\alpha+1}_{H^1}\rightarrow0,
            \end{split}
\end{align*}
and
\begin{align*}
            \begin{split}
               \|\nabla E_3\|_{L^1_{(|t|\leq T)}L^2_x}\lesssim& T\||x|^{-b-1}\|_{L^\infty_{(|x|>\frac{|x_n|}{4})}}\|e^{it\Delta}P_n\psi\|^{\alpha+1}_{L^\infty_tH^1_x}\\
               &+T\||x|^{-b}\|_{L^\infty_{(|x|>\frac{|x_n|}{4})}}\|e^{it\Delta}P_n\psi\|^{\alpha}_{L^\infty_tL^{\infty}_x}\|\nabla e^{it\Delta}P_n\psi\|_{L^\infty_tL^2_x}\\
               \lesssim&T\left(|x_n|^{-b}+|x_n|^{-b-1}\right)\|\psi\|^{\alpha+1}_{H^1}\rightarrow0,
            \end{split}
\end{align*}
furthermore
\begin{align*}\begin{split}
           \|E_3\|_{S'((|t|\leq T),\dot{H}^{-s_c})}\leq  \|E_3\|_{L^{\tilde{q}'}_{(|t|\leq T)}L^2_x}\lesssim& T^\frac{1}{\tilde{q}'}\||x|^{-b}\|_{L^\infty_{(|x|>\frac{|x_n|}{4})}}\|e^{it\Delta}P_n\psi\|^{\alpha+1}_{L^\infty_tL^{2(\alpha+1)}_x}\\
             \lesssim& T^\frac{1}{\tilde{q}'}|x_n|^{-b}\|\psi\|^{\alpha+1}_{H^1}\rightarrow0,
             \end{split}
        \end{align*}
where $\tilde{q}=\frac{2}{s_c}$ so that $(\tilde{q},2)\in\mathcal{A}_{-s_c}$.

If $|t|>T$, we only focus on the region $t>T$, because on the region $t<-T$, the discussion can be addressed similarly. Assuming $t>T$, we have
\begin{align*}
            e_{n,T}=-|x|^{-b}|\tilde{v}_{n,T}|^\alpha\tilde{v}_{n,T},
\end{align*}
and
\begin{align*}
            \tilde{v}_{n,T}(t,x)=e^{i(t-T)\Delta}[\chi_n(x-x_n)e^{iT\Delta}P_n\psi(x-x_n)-\chi_n(-x-x_n)e^{iT\Delta}P_n\psi(-x-x_n)].
\end{align*}
To estimate the time-space norms of $e_{n,T}$, we should use the oddness of both $e_{n,T}$ and $\tilde{v}_{n,T}$ with respect to $x$. By applying Lemma \ref{lem nlest}, we get
\begin{align*}
\begin{split}
            &\|e_{n,T}\|_{S'((t>T),\dot{H}^{-s_c})}+\|e_{n,T}\|_{S'((t>T),L^2)}+\|\nabla e_{n,T}\|_{S'((t>T),L^2)}\\
            \lesssim&\|\tilde{v}_{n,T}\|^\theta_{L^\infty_tH^1_x}\|\tilde{v}_{n,T}\|^{\alpha-\theta}_{S(\dot{H}^{s_c})}\left(\|\tilde{v}_{n,T}\|_{S((t>T),\dot{H}^{s_c})}+\|\tilde{v}_{n,T}\|_{S((t>T),L^2)}+\|\nabla\tilde{v}_{n,T}\|_{S((t>T),L^2)}\right).
\end{split}
\end{align*}
Since both $P_n,~\chi_n$ and $e^{i\cdot\Delta}$ are uniformly bounded operators from $H^1$ to $H^1$, we can conclude that $\|\tilde{v}_{n,T}\|_{L^\infty_tH^1_x}\lesssim\|\psi\|_{H^1}$. To estimate the remainder terms, we fix $(q,r)\in\mathcal{A}_{0}$ and $(\hat{q},\hat{r})\in\mathcal{A}_{s_c}$, so that we have
\begin{align}\label{M:e:3}
\|\nabla\tilde{v}_{n,T}\|_{L^q_{(t>T)}L^r_x}&\lesssim\|\nabla e^{it\Delta}\chi_ne^{iT\Delta}P_n\psi\|_{L^q_{(t>0)}L^r_x}\notag\\
&\lesssim\|\nabla e^{it\Delta}P_n\psi\|_{L^q_{(t>T)}L^r_x}+\|\nabla[(1-\chi_n)e^{iT\Delta}P_n\psi]\|_{L^2},
\end{align}
and
\begin{align}\label{M:e:4}  \|\tilde{v}_{n,T}\|_{L^q_{(t>T)}L^r_x}\lesssim\|e^{it\Delta}P_n\psi\|_{L^q_{(t>T)}L^r_x}+\|(1-\chi_n)e^{iT\Delta}P_n\psi\|_{L^2}
\end{align}
moreover
\begin{align}\label{M:e:5}
\|\tilde{v}_{n,T}\|_{L^{\hat{q}}_{(t>T)}L^{\hat{r}}_x}&\lesssim\|e^{it\Delta}\chi_ne^{iT\Delta}P_n\psi\|_{L^{\hat{q}}_{(t>0)}L^{\hat{r}}_x}\notag\\
            &\lesssim\|e^{it\Delta}P_n\psi\|_{L^{\hat{q}}_{(t>T)}L^{\hat{r}}_x}+\|(1-\chi_n)e^{iT\Delta}P_n\psi\|_{\dot{H}^{s_c}}.
        \end{align}
Using Lemma \ref{str est} again, we  deduce that $\|\nabla e^{it\Delta}P_n\psi\|_{L^q_tL^r_x}\lesssim\|\nabla \psi\|_{L^2}$, $\| e^{it\Delta}P_n\psi\|_{L^q_tL^r_x}\lesssim\|\psi\|_{L^2}$ and $\|e^{it\Delta}P_n\psi\|_{L^{\hat{q}}_tL^{\hat{r}}_x}\lesssim\|\psi\|_{\dot{H}^{s_c}}$.
Consequently, in \eqref{M:e:3}, \eqref{M:e:4}, and \eqref{M:e:5}, the first term approaches zero as $T\rightarrow\infty$, while the second term approaches zero as $n\rightarrow\infty$. Thus, estimate $iii$) holds.

Now, applying Lemma \ref{prop sta}, we conclude that there exists a global solution $v_n$ to  \eqref{eqinls} satisfying $v_n(0)=e^{-it_n\Delta}\Psi_n$. Additionally, for all sufficiently large $n$, we have
$$\|v_n\|_{S(\dot{H}^{s_c})}+\|v_n\|_{S(L^2)}+\|\nabla v_n\|_{S(L^2)}\lesssim 1.$$
Finally, the proof of the approximation aligns completely with that in \cite{MMZ2021}. For further details, we refer to \cite[Proposition 3.3]{MMZ2021}.
\end{proof}

\section{Proof of Theorem \ref{T2}}

In this section, we will complete the proof of Theorem \ref{T2} through a contradiction argument.


\begin{lemma}\label{tightness}
 Let $u$ be a solution to \eqref{eqinls} such that the obit $\{u(t):~t\in\R\}$ is pre-compact in $H^1(\R)$. Then, for any $\varepsilon>0$, there exists $R(\varepsilon)>0$ such that for all $t\in\R$,  the following inequality holds:
\begin{align*}
    \int_{|x|>R}|\nabla u(x,t)|^2+|u(x,t)|^2dx\leq\varepsilon.
\end{align*}
\end{lemma}
\begin{proof}
If Lemma \ref{tightness} fails, then there exists  $A>0$ and a sequence $\{t_n\}$ of time,  such that for all $R>0$,
\begin{align*}
             \int_{|x|>R}|\nabla u(x,t_n)|^2+|u(x,t_n)|^2dx\geq A.
\end{align*}
Since the orbit $\{u(t):~t\in\R\}$ is pre-compact, there exists $\varphi\in H^1(\R)$ such that, up to a subsequnce,
\begin{align*}
            u(t_n)\rightarrow\varphi~~\text{in}~~H^1(\R).
\end{align*}
This implies that for all sufficiently large $R>0$,
\begin{align*}
\int_{|x|>R}|\nabla\varphi(x)|^2+|\varphi(x)|^2dx>\frac{A}{2}.
\end{align*}
This leads to a contradiction with the fact that $\varphi\in H^1(\R)$.
\end{proof}


\begin{lemma}\label{vanish}
If a global odd solution $u$ to \eqref{eqinls} with initial data $u_0$ satisfies that the orbit $\{u(t):~t\in\R\}$ is pre-compact in $H^1(\R)$, then $u_0\equiv0$.
\end{lemma}
\begin{proof}
Applying the mass conservation \eqref{M} and Lemma \ref{tightness}, we have 
    \begin{align*}
        \int_\R|u_0|^2dx=\int_\R |u(t)|^2dx=\int_{|x|\leq R}|u(t)|^2dx+\int_{|x|>R}|u(t)|^2dx\leq\int_{|x|\leq R(\varepsilon)}|u(t)|^2dx+\varepsilon.
    \end{align*}
Choosing $\varepsilon$ sufficiently small and $t>T(R(\varepsilon))$ large, by Lemma \ref{lem decay}, we get 
\begin{align*}
        \int_{|x|\leq R(\varepsilon)}|u(t)|^2dx<\varepsilon.
\end{align*}
The existence of $T(R(\varepsilon))$ arises from the fact that, for any bounded space interval $I$, $$\lim_{t\rightarrow\infty}\|u(t)\|_{L^2(I)}\rightarrow0.$$

As a result, we get $\|u_0\|_{L^2}=0$, which implies that $u\equiv0$.
\end{proof}

\begin{proof}[{\bf Proof of Theorem \ref{T2}}]
Proposition \ref{prop c} and Lemma \ref{vanish} show that the non-scattering solution always vanishes, which is obviously impossible. Therefore, the proof of Theorem \ref{T2} is completed.
\end{proof}

\end{document}